\documentclass[twoside]{article}
\usepackage{enumerate,mathrsfs,amsfonts,amsxtra,latexsym,amssymb,amsthm,verbatim,amsmath,graphicx,dsfont,extarrows}
\usepackage{mathpazo}
\usepackage[dvipdfm, colorlinks, linkcolor=green, anchorcolor=blue, citecolor=red]{hyperref}
\allowdisplaybreaks[4]
\catcode`@=12
\newtheorem{theorem}{Theorem}[section]
\newtheorem{definition}[theorem]{Definition}
\newtheorem{lemma}[theorem]{Lemma}
\newtheorem{corollary}[theorem]{Corollary}
\newtheorem{proposition}[theorem]{Proposition}
\begin{document}
\abovedisplayskip=6pt plus 1pt minus 1pt \belowdisplayskip=6pt
plus 1pt minus 1pt
%-------------------  First Head  -----------------------------------------
\thispagestyle{empty} \vspace*{-1.0truecm} \noindent
\vskip 10mm

\begin{center}{\large Bowling ball representation of virtual string links} \end{center}

\vskip 5mm
\begin{center}{Zhiyun Cheng\\
{\small School of Mathematical Sciences, Beijing Normal University
\\Laboratory of Mathematics and Complex Systems, Ministry of
Education, Beijing 100875, China
\\(email: czy@bnu.edu.cn)}}\end{center}

\vskip 1mm

\noindent{\small {\small\bf Abstract} In this paper we investigate the virtual string links via a probabilistic interpretation. This representation can be used to distinguish some virtual string links from classical string links. In order to study the algebraic structure behind this probabilistic interpretation we introduce the notion of virtual flat biquandle. The cocycle invariants associated with virtual flat biquandle is discussed.
\ \

\vspace{1mm}\baselineskip 12pt


\noindent{\small\bf MR(2010) Subject Classification} 57M25, 57M27\ \ {\rm }}

\vskip 1mm

\vspace{1mm}\baselineskip 12pt

\let\thefootnote\relax\footnote{The author is supported by NSFC 11301028 and the Fundamental Research Funds for Central Universities of China 105-105580GK.}

\section{Introduction}
In his seminal paper \cite{Jon1987}, Jones mentioned a probabilistic interpretation of the unreduced Burau representation of positive braids. This interpretation was generalized by Xiaosong Lin, Feng Tian and Zhenghan Wang \cite{Lin1998} to a representation of the monoid of string links. More precisely they assigned each string link a Burau matrix. As an application, in \cite{Lin2001} Xiaosong Lin and Zhenghan Wang used this model of random walks on knot diagrams to give an alternative proof of the Melvin-Morton conjecture, which was first settled by D. Bar-Natan and S. Garoufalidis in \cite{Bar1996} via finite type invariants. Later this Burau matrix was generalized to a 2-variable Burau matrix by Daniel S. Silver and Susan G. Williams in \cite{Sil2001}. Recently by considering the case of several bowling balls on a knot diagram simultaneously \cite{Big2014}, Stephen Bigelow obtained a cabled version of the Temperley-Lieb representation.

In this paper, we study the virtual string links by bowling a ball along the tangled lanes. We consider all the possibilities when the ball meets a $($real or virtual$)$ crossing point, then we discuss the constrains deduced from the generalized Reidemeister moves. We remark that in general this is not a realistic probability model except some very special cases. Similar to the Burau matrix introduced in \cite{Lin1998}, in Section 2 each virtual string link is associated with an $n\times n$ matrix and each entry of it takes the value in $\mathds{Z}[s]/(s^2)$, here $n$ denotes the number of strands of the link. Next we study the general algebraic structure of this probabilistic interpretation and introduce the notion of virtual flat biquandle, which is closely related to the flat biquandle studied in \cite{Hen2010} and \cite{Kau2012}. Finally an enhancement of the counting invariant associated with virtual flat biquandle is discussed in Section 5.

\section{Bowling ball representation of virtual string links}
By a \emph{virtual $n$-string link diagram}, we mean a collection of $n$ immersed strings in the strip $\mathds{R}\times [0, 1]$ such that the $i$-th string gives an oriented path from $(i, 1)$ to $(\pi(i), 0)$, here $\pi$ denotes a permutation of $\{1, \cdots, n\}$. Each crossing of this diagram is either real or virtual. A \emph{virtual $n$-string link} is an equivalence class of virtual $n$-string link diagrams under generalized Reidemeister moves, see Figure 1. Similar to the virtual knots which is an extension of classical knots \cite{Kau1999}, virtual string links can be regarded as the virtual version of the classical string links \cite{Lin1998}. Obviously the set of all virtual $n$-string links has a monoid structure. In particular, if each strand meets $\mathds{R}\times t$ $(0<t<1)$ transversely at one point then we obtain the virtual braid $VB_n$ \cite{Kau1999,Kam2007}.
\begin{center}
\includegraphics{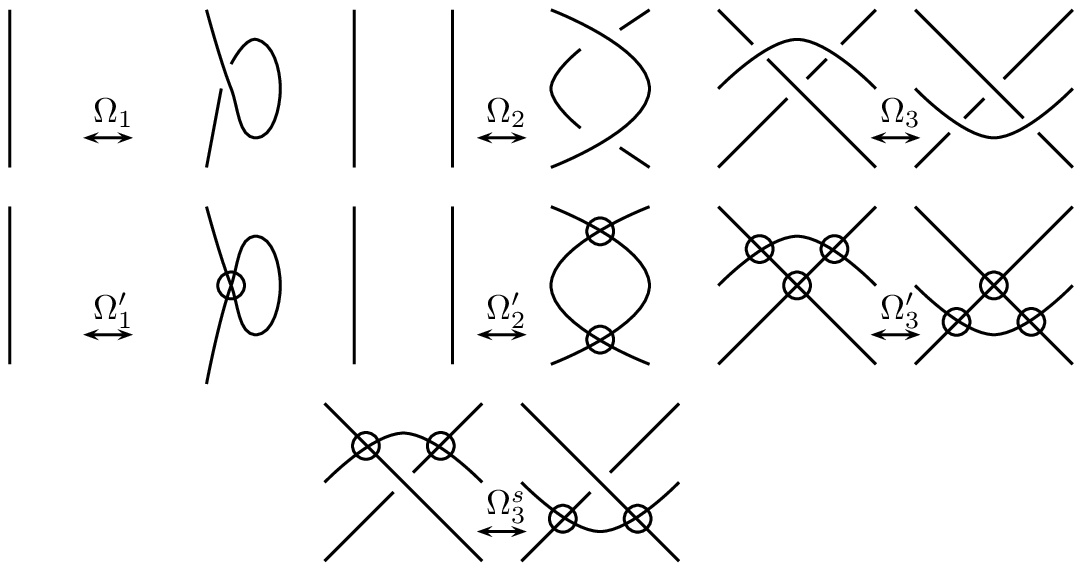}
\centerline{\small Figure 1}
\end{center}

We give an example of virtual 2-string link in Figure 2.
\begin{center}
\includegraphics{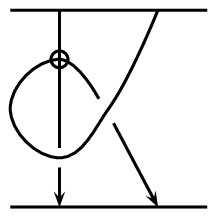}
\centerline{\small Figure 2}
\end{center}

Now we state the bowling ball representation of virtual string links. As we mentioned before, although for most cases the model we describe below is not a realistic probability model, we will still use the word ``probability"$($or weight$)$ for convenience. With a fixed virtual $n$-string link diagram $L$, we can imagine it as a bowling alley with $n$ lanes. If we put a bowling ball at $(i, 1)$ $(i\in \{1, \cdots, n\})$, then this ball will travel along the lane according to the orientation and behave according to the following rules:
\begin{enumerate}
  \item If we come to a negative crossing on the upper lane, the ball jumps down with probability $1-t$ and keeps walking with probability $t$.
  \item If we come to a negative crossing on the lower lane, the ball jumps up with probability $1-u$ and keeps walking with probability $u$.
  \item If we come to a positive crossing on the lower lane, the ball jumps up with probability $1-w$ and keeps walking with probability $w$.
  \item If we come to a positive crossing on the upper lane, the ball jumps down with probability $1-v$ and keeps walking with probability $v$.
  \item If we come to a virtual crossing from the left side, the ball jumps to the other lane with probability $s$ and keeps walking with probability $1-s$
  \item If we come to a virtual crossing from the right side, the ball jumps to the other lane with probability $r$ and keeps walking with probability $1-r$
\end{enumerate}

We illustrate all these possibilities in Figure 3.
\begin{center}
\includegraphics{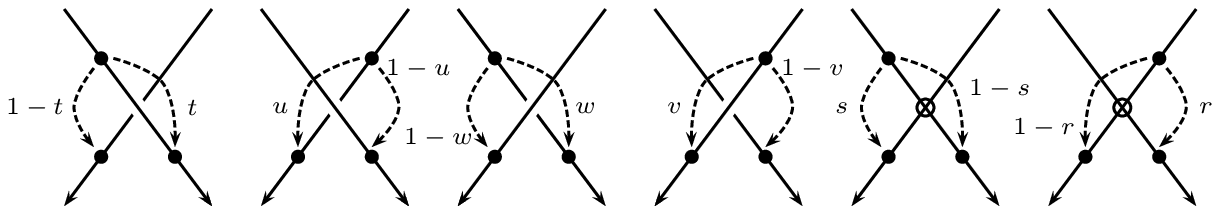}
\centerline{\small Figure 3}
\end{center}

Now for a fixed virtual $n$-string link diagram $L$, we can associate a matrix $M(L)$ with $L$ by defining the $(i, j)$-th entry to be the probability that a ball begins in the $i$-th lane and ends up in the $j$-th lane. In order to make $M(L)$ does not depend on the choice of the diagram, we need to study the constrains deduced from the generalized Reidemeister moves.

$\Omega_2:$
\begin{center}
\includegraphics{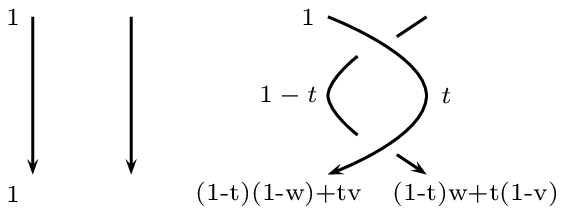}
\centerline{\small Figure 4}
\end{center}
From Figure 4 we conclude that
\begin{center}
$\begin{cases}
(1-t)(1-w)+tv=1\\
(1-t)w+t(1-v)=0,
\end{cases}$
\end{center}
which follows by
\begin{equation}
(w+v-1)t=w.
\label{1}
\end{equation}
Similarly if we put the ball on the lower lane of Figure 4, we will obtain that
\begin{equation}
(w+v-1)u=v.
\label{2}
\end{equation}

$\Omega_3:$
\begin{center}
\includegraphics{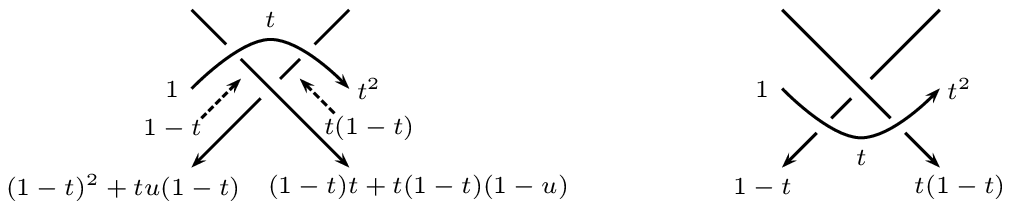}
\centerline{\small Figure 5}
\end{center}
From Figure 5 we conclude that
\begin{center}
$\begin{cases}
(1-t)^2+tu(1-t)=1-t\\
(1-t)t+t(1-t)(1-u)=t(1-t),
\end{cases}$
\end{center}
which follows that
\begin{equation}
t(1-t)(1-u)=0.
\label{3}
\end{equation}
Note that if we reverse the orientation of the horizontal string and switch the two crossings of it, then the new diagram will give us
\begin{equation}
u(1-t)(1-u)=0.
\label{4}
\end{equation}

$\Omega_2':$
\begin{center}
\includegraphics{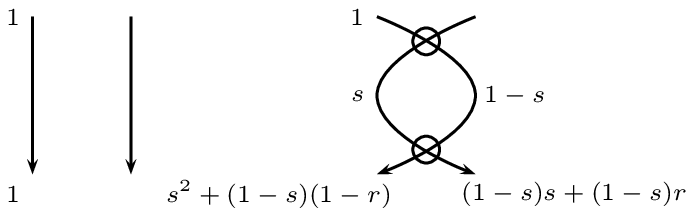}
\centerline{\small Figure 6}
\end{center}
From Figure 6 we conclude that
\begin{center}
$\begin{cases}
s^2+(1-s)(1-r)=1\\
(1-s)s+(1-s)r=0,
\end{cases}$
\end{center}
hence we have
\begin{equation}
(s-1)(s+r)=0.
\label{5}
\end{equation}
If we put the ball on the other lane of Figure 6, we will obtain that
\begin{equation}
(r-1)(s+r)=0.
\label{6}
\end{equation}

From \eqref{3} and \eqref{4} we have $t=1$ or $u=1$ or $t=u=0$, on the other hand \eqref{5}\eqref{6} tell us $s=r=1$ or $s=-r$. We continue our discussion according to the following six cases:
\begin{enumerate}
\item $t=1$ and $s=r=1$. Together \eqref{1} with $t=1$, we conclude that $v=1$. Now let us consider the generalized Reidemeister move $\Omega_3^s$ described below.
\begin{center}
\includegraphics{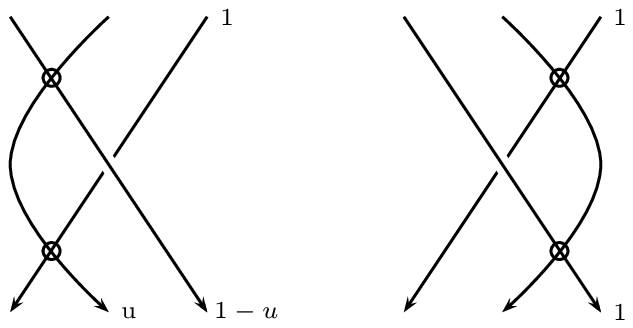}
\centerline{\small Figure 7}
\end{center}
It can be read from Figure 7 that $u$ has to be 0. Similarly, by switching the real crossing in Figure 7 we have $w=0$. According to the analysis above now we find that $u=w=0$ and $v=1$, however this contradicts with \eqref{2}.
\item $u=1$ and $s=r=1$. Since $(w+v-1)u=v$ and $u=1$, it follows that $w=1$. Now Figure 8 tells us that $t=0$, as before if we switch the real crossing in Figure 8 we obtain that $v=0$. Now the equality \eqref{1} does not hold anymore.
\begin{center}
\includegraphics{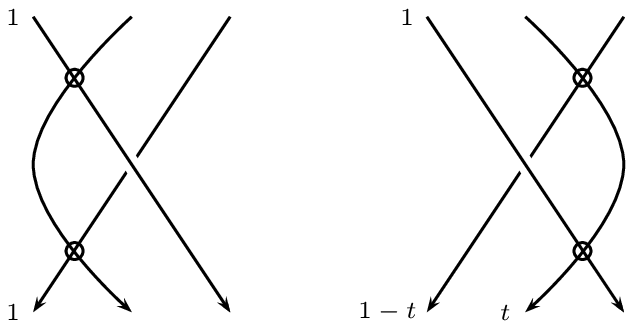}
\centerline{\small Figure 8}
\end{center}
\item $t=u=0$ and $s=r=1$. In this case we have $w=v=0$, because of equalities \eqref{1} and \eqref{2}. Let us consider another $\Omega_3$ move, which is a bit different from that in Figure 5. See the figure below.
\begin{center}
\includegraphics{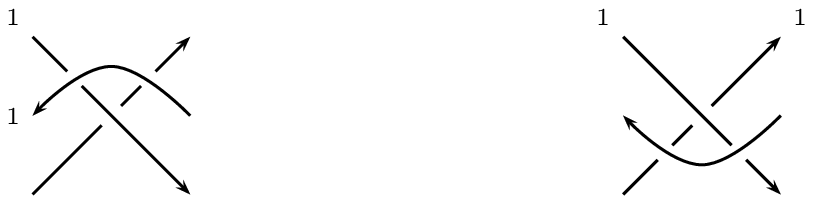}
\centerline{\small Figure 9}
\end{center}
Here the ball is placed on the top left lane at first, but its direction is not preserved under $\Omega_3$.
\item $t=1$ and $s=-r$. This case is essentially equivalent to the fifth case (after switching real crossings), hence we omit it here.
\item $u=1$ and $s=-r$. Analogous to the second case, the equalities \eqref{1}\eqref{2} and $u=1$ implies that $w=1$ and $vt=1$. Now let us go back to Figure 8 and put a ball in the middle lane, the outcome is illustrated in Figure 10.
\begin{center}
\includegraphics{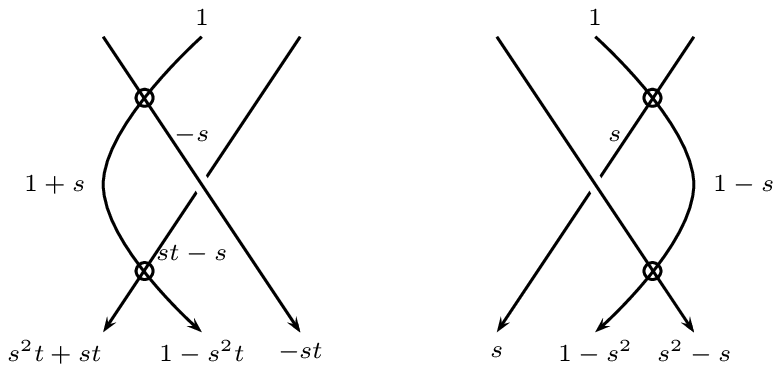}
\centerline{\small Figure 10}
\end{center}
It can be read directly from Figure 10 that
\begin{equation}
\begin{cases}
s(st+t-1)=0\\
s^2(t-1)=0\\
s(s+t-1)=0
\end{cases}
\label{7}
\end{equation}
We will come back to \eqref{7} after finishing the last case.
\item $t=u=0$ and $s=-r$.  Due to the same reason of the third case, there is nothing interesting in this case.
\end{enumerate}

According to the discussion above the only interesting case is the fifth one, note that the fourth case is essentially equivalent to this one. Now we have $u=w=1$, $vt=1$, $s=-r$ and the equalities \eqref{7}. Under these assumptions we turn to the generalized Reidemeister move $\Omega_3'$, see the figure below.
\begin{center}
\includegraphics{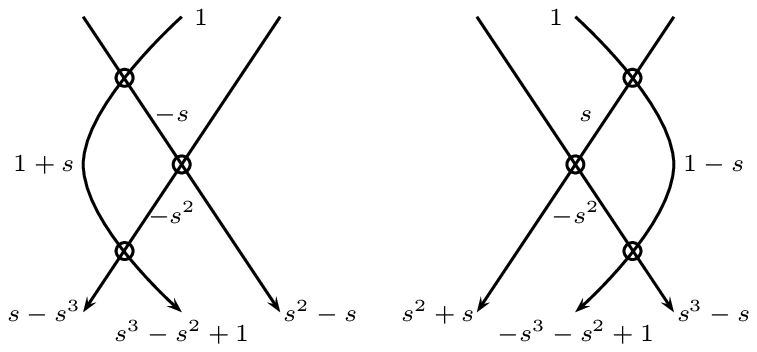}
\centerline{\small Figure 11}
\end{center}
From Figure 11 we conclude that
\begin{center}
$\begin{cases}
s-s^3=s^2+s\\
s^3-s^2+1=-s^3-s^2+1\\
s^2-s=s^3-s,
\end{cases}$
\end{center}
which implies that $s^2=0$. Now equalities \eqref{7} turn out to be $s(t-1)=0$. Finally we have two choices of the bowling ball models:
\begin{enumerate}
  \item $u=w=1$, $s=r=0$ and $v=t^{-1}$,
  \item $u=w=t=v=1$, $r=-s$ and $s^2=0$.
\end{enumerate}

The reader who is familiar with the Burau representation must have found that the first case is nothing but the Burau representation discussed in \cite{Lin1998}. In particular when the virtual string link diagram is a classical braid then this is exactly the classical Burau representation. When the virtual string link diagram contains no virtual crossings, it was proved that each entry of $M(L)$ converges to a rational function of $t$ and the matrix $M(L)$ is invariant under Reidemeister moves \cite{Lin1998}. If the diagram contains some virtual crossings, since $s=r=0$, when we come to a virtual crossing the ball will keep walking. It is not difficult to prove that in this case $M(L)$ is also well-defined and preserved under generalized Reidemeister moves. We remark that the generalization of this representation introduced in \cite{Sil2001} is also valid for virtual string links.

In the remainder of this paper we will focus on the second case, i.e. $u=w=t=v=1$, $r=-s$ and $s^2=0$. Therefore each entry of $M(L)$ takes the value in $\mathds{Z}[s]/(s^2)$. Now we have the following theorem.
\begin{theorem}
Let $L$ be a virtual $n$-string link diagram, then we can assign an $n\times n$ matrix $M(L)$ to $L$ such that
\begin{enumerate}
  \item Each entry of $M(L)$ has the form $as+b$, here $a\in\mathds{Z}$ and $b\in\{0,1\}$;
  \item $M(L)$ is invariant under generalized Reidemeister moves. Moreover $M(L)$ determines a representation of the monoid of virtual $n$-string links.
\end{enumerate}
\end{theorem}
\begin{proof}Before proving the theorem, we recall some terminology in \cite{Lin1998}. Assume a ball begins at $(i,1)$ and ends up at $(j,0)$, we call the way of walking a \emph{path}. A \emph{loop} is a part of a path that begins and ends at the same crossing. If a path $($loop$)$ contains no loops $($except itself$)$ then we say it is \emph{simple}. The \emph{multiplicity} of a path is the number of simple loops that it contains. Obviously for a given path its multiplicity is finite.
\begin{enumerate}
  \item Let us consider the $(i,j)$-th entry of $M(L)$, the key point is that there are only finite paths connecting $(i,1)$ and $(j,0)$. We remark that in the first case, i.e. $u=w=1$, $s=r=0$ and $v=t^{-1}$, maybe there are countably many paths walking from $(i,1)$ to $(j,0)$. Starting at the point $(i,1)$, the following figure indicates all the possibilities when we meet a virtual crossing point, note that nothing happens when we cross a real crossing.
\begin{center}
\includegraphics{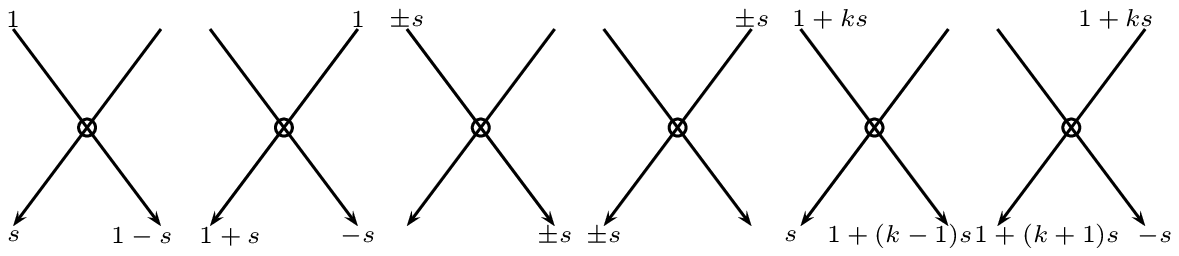}
\centerline{\small Figure 12}
\end{center}
Choose a simple loop $l$, we claim that for any path $p$, the multiplicity of $l$ in $p$ is at most two. First we notice that when we meet a virtual crossing for the first time, the ball will split into two balls with weights $1\pm s$ and $\mp s$ respectively. From Figure 12 we observe that any ball with weight $\pm s$ will walk straight until it arrives $(j,0)$.

For the ball with weight $1\pm s$, when it meets the next virtual crossing, it will continue splitting into two balls, one with weight $\mp s$, the other with weight $(1\pm s)\pm s$. As we have discussed above, the ball with weight $\mp s$ will go to $(j,0)$ straightly, hence we do not need to consider it.

Let us consider a ball $B$ with weight $1+ks$, and we denote the virtual crossing in front of $B$ by $c$. Without loss of generality we assume that $B$ locates on the left top corner of $c$. When we meet $c$, the ball $B$ will split into two balls $B_1$ and $B_2$, with weights $s$ and $1+(k-1)s$ respectively. Since $B_1$ will walk to $(j,0)$ directly, it suffices to consider the other ball, $B_2$. When $B_2$ meets the next virtual crossing it splits into two balls, with weights $\pm s$ and $1+(k-1)s\mp s$ respectively. The ball with weight $\pm s$ may come back to $c$ again, but as we have discussed before it will never jump to other lanes. For this reason it is sufficient to consider the case that $B_2$ comes back to $c$ from the right top corner. Since it may have met some virtual crossings during the loop $l$, the weight of it may have been changed, say $1+k's$. But we still use $B_2$ to denote it. This time $B_2$ will split into two balls, say $B_{2_1}$ and $B_{2_2}$. Here $B_{2_1}$ has weight $1+(k'+1)s$ and $B_{2_2}$ has weight $-s$. Obviously $B_{2_2}$ will continue its journey along $l$ and come back to $c$ again. After that it walks to $(j,0)$ straightly. As a consequence, the claim above is proved. The first part of the theorem follows accordingly.
\begin{center}
\includegraphics{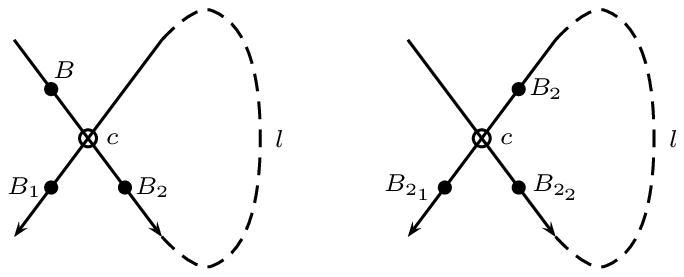}
\centerline{\small Figure 13}
\end{center}
  \item It suffices to check those moves which contain virtual crossings. Some of them have been checked in Figure 6, Figure 10 and Figure 11. We illustrate the invariance under $\Omega_1'$ and another $\Omega_2'$ in Figure 14. The other cases of $\Omega_3'$ and $\Omega_3^s$ can be verified by an analogous argument.
\begin{center}
\includegraphics{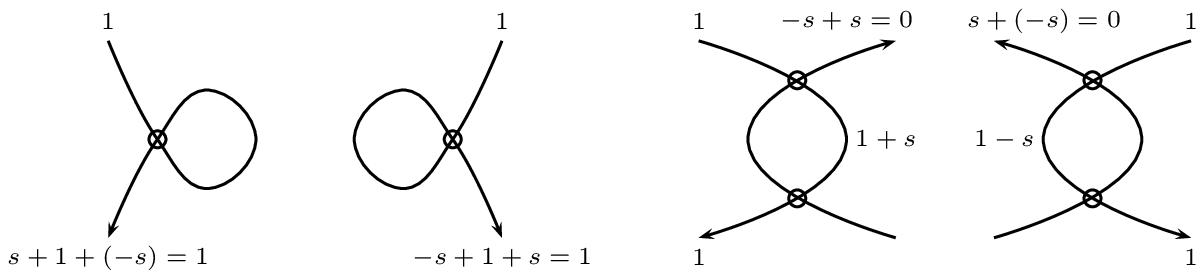}
\centerline{\small Figure 14}
\end{center}
\end{enumerate}
\end{proof}

\textbf{Remark} According to the proof of Theorem 2.1, it turns out that for any $1\leq i\leq n$ the entries of $M(L)=(m_{ij})_{n\times n}$ have the form
\begin{center}
$\begin{cases}
m_{ij}=a_{ij}s+1, &\text{if } j=\varphi(i);\\
m_{ij}=a_{ij}s, &\text{otherwise},
\end{cases}$
\end{center}
here $\varphi:\{1, \cdots, n\}\rightarrow\{1, \cdots, n\}$ denotes the permutation induced from the virtual string link. In particular, if $L$ contains no virtual crossings, then $a_{ij}=0$ $(1\leq i, j\leq n)$. On the other hand, for any $1\leq i\leq n$ we always have
\begin{center}
$\sum\limits_{j=1}^nm_{ij}=1,$
\end{center}
which implies that $M(L)$ always has 1 as an eigenvalue with an eigenvector
$\begin{pmatrix}
1 \\
\vdots \\
1 \\
\end{pmatrix}.$

\section{Some examples}
In this section we give some examples of virtual string link and calculate the associated matrix invariants. Note that if $L$ is the trivial $n$-string link, i.e. there exists no crossing points, then $M(L)=I_n$. On the other hand if some entry $m_{ij}=a_{ij}s+b_{ij}$ of $M(L)$ has nonzero $a_{ij}$ then it follows that $L$ is nonclassical.

First let us consider the virtual string link $L$ illustrated in Figure 2. Direct calculation shows that
\begin{center}
$M(L)=\begin{pmatrix}
1+s & -s \\
s & 1-s \\
\end{pmatrix},$
\end{center}
hence this virtual string link is nonclassical and hence nontrivial.

Next let us consider the following virtual 2-string links:
\begin{center}
\includegraphics{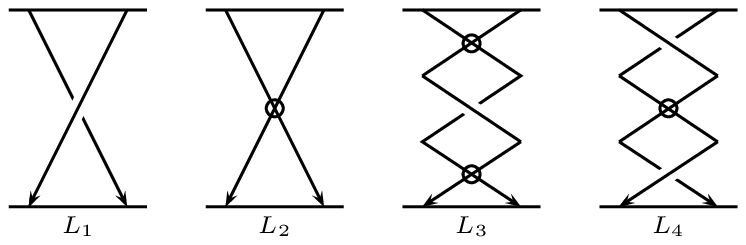}
\centerline{\small Figure 15}
\end{center}
Routine calculation shows that
\begin{center}
$M(L_1)=\begin{pmatrix}
0 & 1 \\
1 & 0 \\
\end{pmatrix},
M(L_2)=\begin{pmatrix}
s & 1-s \\
1+s & -s \\
\end{pmatrix},
M(L_3)=\begin{pmatrix}
2s & 1-2s \\
1+2s & -2s \\
\end{pmatrix},
M(L_4)=\begin{pmatrix}
-s & 1+s \\
1-s & s \\
\end{pmatrix},$
\end{center}
which implies that they are mutually different.

Suppose $L=K_1\cup K_2$ is a virtual 2-string link, here $K_1$ and $K_2$ denote the two strings. We use $RC_{K_1\cap K_2}$ and $VC_{K_1\cap K_2}$ to refer to the set of real crossings and the set of virtual crossings between $K_1$ and $K_2$ respectively. One can easily check that
\begin{center}
$lk(L)=\sum\limits_{c\in RC_{K_1\cap K_2}}w(c)$ and $lk_v(L)=\sum\limits_{c\in VC_{K_1\cap K_2}}1$ $($mod 2$)$
\end{center}
are both invariant under generalized Reidemeister moves. Here $w(c)$ refers to the writhe of the real crossing $c$. For example, one can use $lk(L)$ and $lk_v(L)$ to show that $L_1, L_2 (L_4), L_3$ in Figure 15 are mutually different. However for $L_2$ and $L_4$, we have $lk(L_2)=lk(L_4)$ and $lk_v(L_2)=lk_v(L_4)$. Hence $L_2$ and $L_4$ can not be distinguished by these two invariants, but according to the calculation above they can be distinguished by the matrix invariant.

Before ending this section we give a simple application of the matrix invariant. For virtual 2-string link $L$ we have defined an invariant $lk_v(L)\in \mathds{Z}_2$. In general for a virtual $n$-string link $L=K_1\cup\cdots\cup K_n$ we can also consider the the number of virtual crossings between $K_i$ and other strings. However it is evident that this is not an invariant. Instead, we can consider the minimal number of this for all diagrams that represent $L$. We use $lk_v(L; K_i)$ to denote it. From some point $lk_v(L; K_i)$ measures the ``virtual linking complexity" between $K_i$ and other strings. Assume the assigned matrix $M(L)=(m_{ij})_{n\times n}=(a_{ij}s+b_{ij})_{n\times n}$, we define $a_i(L)=$max$|a_{ij}|$ for all $1\leq j\leq n$.
\begin{proposition}
$lk_v(L; K_i)\geq a_i(L)$.
\end{proposition}
\begin{proof}
The result mainly follows from Figure 12. Notice that the contribution of a virtual self-crossing to $a_i(L)$ cancels out, and a virtual non-self-crossing has contribution $\pm1$ to $a_i(L)$.
\end{proof}

\section{Virtual flat biquandle}
In section 2 we give a bowling ball representation of virtual $n$-string link. In particular this interpretation also offers a representation of the virtual braid group $VB_n$. Recall that $VB_n$, the group of virtual braids on $n$ strings is generated by $\sigma_1, \cdots, \sigma_{n-1}$ and $\tau_1, \cdots, \tau_{n-1}$. Here $\sigma_i$ and $\tau_i$ correspond to the positive crossing and virtual crossing between the $i$-th string and $(i+1)$-th string respectively. The relations are listed below:
\begin{enumerate}
  \item $\sigma_i\sigma_j=\sigma_j\sigma_i$, if $|i-j|>1$;
  \item $\sigma_i\sigma_{i+1}\sigma_i=\sigma_{i+1}\sigma_i\sigma_{i+1}$;
  \item $\tau_i^2=1$;
  \item $\tau_i\tau_j=\tau_j\tau_i$, if $|i-j|>1$;
  \item $\tau_i\tau_{i+1}\tau_i=\tau_{i+1}\tau_i\tau_{i+1}$;
  \item $\sigma_i\tau_j=\tau_j\sigma_i$, if $|i-j|>1$;
  \item $\sigma_i\tau_{i+1}\tau_i=\tau_{i+1}\tau_i\sigma_{i+1}$.
\end{enumerate}
Define a homomorphism $\rho:VB_n\rightarrow$ GL$_n(\mathds{Z}[s]/(s^2))$ as follows:
\begin{center}
$\sigma_i\rightarrow I_{i-1}\oplus\begin{pmatrix}
0 & 1 \\
1 & 0 \\
\end{pmatrix}\oplus I_{n-i-1}$ and $\tau_i\rightarrow I_{i-1}\oplus\begin{pmatrix}
s & 1+s \\
1-s & -s \\
\end{pmatrix}\oplus I_{n-i-1}$.
\end{center}
We remark that $\rho$ is evidently not faithful, for example $\rho(\sigma_1^2)=
\begin{pmatrix}
1 & 0 \\
0 & 1 \\
\end{pmatrix}$.
\begin{theorem}
If $\beta\in VB_n$, then $M(\beta)^T=\rho(\beta)$.
\end{theorem}
\begin{proof}
Notice that if $\beta$ is a virtual braid, then each path contains no loops. The result follows directly from the action of real crossing and virtual crossing, see the figure below.
\begin{center}
\includegraphics{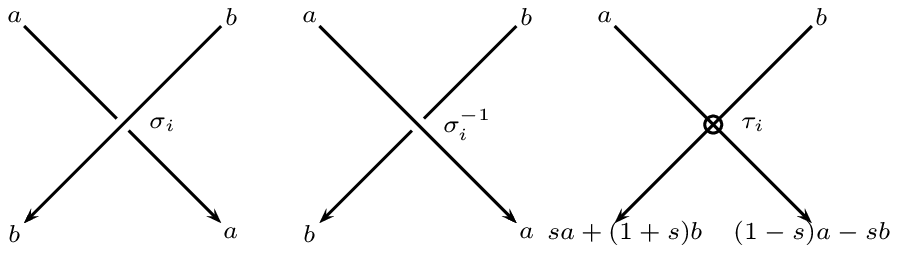}
\centerline{\small Figure 16}
\end{center}
\end{proof}

Motivated by Figure 16, in this section we want to discuss the root structure of the representation $\rho$. Recall that a \emph{quandle} $(Q, \ast)$, is a set $Q$ with a binary operation $(a, b)\rightarrow a\ast b$ satisfying the following axioms:
\begin{enumerate}
  \item $\forall a\in Q$, $a\ast a=a$.
  \item $\forall b, c\in Q$, $\exists!a\in Q$ $($there exists an unique $a\in Q)$ such that $a\ast b=c$.
  \item $\forall a, b, c\in Q$, $(a\ast b)\ast c=(a\ast c)\ast(b\ast c)$.
\end{enumerate}
The notion of quandle was first introduced by Joyce \cite{Joy1982} and Matveev\cite{Mat1984} independently. For each classical knot there is an associated knot quandle, which is known to be a powerful invariant. With a fixed finite quandle $Q$, one can count the number of homomorphisms from the knot quandle to $Q$, which is known as the quandle counting invariant. From the viewpoint of knot diagram, this is equivalent to count the colorings which assign an element of $Q$ to each arc of the diagram, such that some condition $($see Figure 17$)$ is satisfied at each crossing point. Here the term arc means a part of the diagram from an undercrossing to the next undercrossing. Instead of coloring arcs, one also can color the semiarcs of a knot diagram, here a semiarc denotes a part of the diagram from a crossing to the next crossing. For this purpose, the notion of \emph{biquandle} was proposed in \cite{Fen2004}. Later, for virtual knots L. Kauffman and V. O. Manturov \cite{Kau2005} introduced the notion of \emph{virtual biquandle}, which assigns some relations on virtual crossing points $($see Figure 17$)$. We refer the readers to the references mentioned above for a detailed definition of these quandle structures.
\begin{center}
\includegraphics{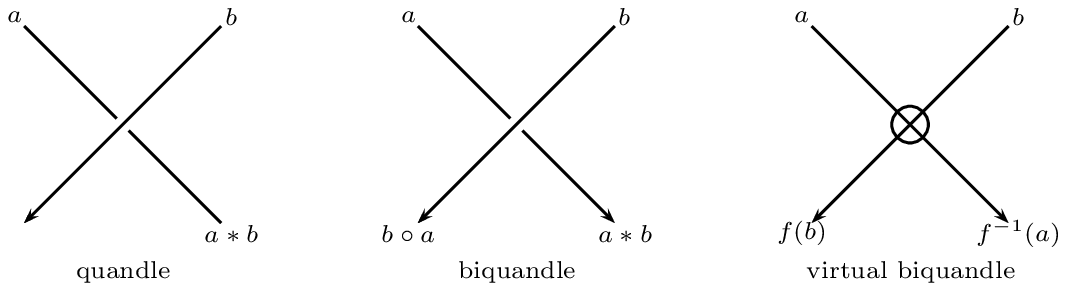}
\centerline{\small Figure 17}
\end{center}

Before introducing the flat biquandle \cite{Kau2012} $($which was named as semiquandle in \cite{Hen2010}$)$ we need to take a brief review of flat virtual knot theory. A flat virtual knot diagram is obtained from a virtual knot diagram by replacing all real crossings by flat crossings. Similarly replacing all the real crossings in generalized Reidemeister moves with flat crossings one obtains the flat Reidemeister moves. Then a \emph{flat virtual knot} is an equivalence class of flat virtual knot diagrams under flat Reidemeister moves. In other words, flat virtual knots can be described as the equivalence classes of virtual knots without under/overcrossing information. Given a flat virtual knot diagram $K$, if $K$ has $k$ flat crossings, then by replacing each flat crossing with an undercrossing or an overcrossing one obtains $2^k$ virtual knot diagrams. According to \cite{Hen2010} we say $K$ is the \emph{shadow} of these virtual knot diagrams and each one of these virtual knot diagrams is a \emph{lift} of $K$. By definition, it is not difficult to find that if $K$ represents a nonclassical flat virtual knot then each lift of it represents a nonclassical virtual knot. Therefore it is significant to detect whether a flat virtual knot is classical or not. The flat biquandle structure plays an important role in flat virtual knot theory. By a \emph{flat biquandle}, we mean a set $FB$ with two binary operations denote $a\ast b$ and $a\circ b$ satisfying the following axioms, which are derived from variations on Reidemeister moves:
\begin{enumerate}
  \item $\forall a\in FB, \exists!x, y\in FB$ such that $a\circ x=x, x\ast a=a, y\circ a=a, a\ast y=y$;
  \item $\forall a, b\in FB, \exists!x, y\in FB$ such that $x=b\circ y, y=a\circ x, b=x\ast a, a=y\ast b$, and $(a\circ b)\ast(b\ast a)=a, (b\ast a)\circ(a\circ b)=b$;
  \item $\forall a, b, c\in FB$, we have $(a\circ b)\circ c=(a\circ(c\ast b))\circ(b\circ c), (c\ast b)\ast a=(c\ast(a\circ b))\ast(b\ast a), (b\circ c)\ast(a\circ(c\ast b))=(b\ast a)\circ(c\ast(a\circ b))$.
\end{enumerate}
Similar to the biquandle coloring described in Figure 17, with a flat biquandle $FB$ one can use it to color a flat virtual knot diagram. Moreover one can also add an operation at virtual crossings like virtual biquandle $($see Figure 17$)$, then one need to add a unary operation to flat biquandle to define the \emph{virtual semiquandle}, see \cite{Hen2010} for more details.

Similar to the flat virtual knots, one can define the flat virtual string links and flat virtual braids. According to our definition of the matrix invariant, it is obvious that the matrix invariant is a flat virtual string link invariant, and the representation $\rho$ can be regarded as a representation of flat virtual braids $FVB_n$. Notice that $FVB_2=\{\sigma_1, \tau_1|\sigma_1^2=\tau_1^2=1\}=\mathds{Z}_2\ast \mathds{Z}_2$. Since
\begin{center}
$\rho((\sigma_1\tau_1)^m)=
\begin{pmatrix}
1-ms & -ms \\
ms & 1+ms \\
\end{pmatrix}$ and
$\rho((\tau_1\sigma_1)^m)=
\begin{pmatrix}
1+ms & ms \\
-ms & 1-ms \\
\end{pmatrix}$,
\end{center}
it follows that $\rho$ is faithful for $FVB_2$.
We remark that this is not true for $FVB_n$ when $n\geq3$. For example,
\begin{center}
$\rho(\sigma_2\tau_1\sigma_1\tau_2\sigma_2\sigma_1\tau_1\sigma_2\tau_2\sigma_2)=
\begin{pmatrix}
1 & 0 & 0 \\
0 & 1 & 0 \\
0 & 0 & 1 \\
\end{pmatrix}$,
\end{center}
however $\sigma_2\tau_1\sigma_1\tau_2\sigma_2\sigma_1\tau_1\sigma_2\tau_2\sigma_2$ is nontrivial. In order to see this, consider the permutation representation $r: FVB_3\rightarrow S_3$ defined by
\begin{center}
$r(\sigma_1)=r(\sigma_2)=(1, 2, 3), r(\tau_1)=(2, 1, 3)$ and $r(\tau_2)=(1, 3, 2)$.
\end{center}
Then $r(\sigma_2\tau_1\sigma_1\tau_2\sigma_2\sigma_1\tau_1\sigma_2\tau_2\sigma_2)=(3, 1, 2)\neq(1, 2, 3)$.

The main idea of virtual flat biquandle is derived from Figure 16. Unlike the biquandle which adds two binary operations at flat crossings but no operation at virtual crossings, we would like to add two binary operations at virtual crossings but no operation at flat crossings. The difference between our construction and virtual biquandle $($or virtual semiquandle$)$ is that virtual biquandle $($or virtual semiquandle$)$ adds one unary operation at virtual crossings, but we add two binary operations.
\begin{definition}
A virtual flat biquandle is a set $VFB$ with two binary operations denoted by $a\ast b$ and $a\circ b$. If we denote $a\ast b$ and $a\circ b$ by $S_b(a)$ and $T_b(a)$ respectively, then $S_a, T_a: VFB\rightarrow VFB$ satisfy the following axioms:
\begin{enumerate}
  \item $S_aS_b=S_bS_a, T_aT_b=T_bT_a, S_aT_b=T_bS_a$;
  \item $S_a=S_{T_b(a)}=S_{S_b(a)}, T_a=T_{S_b(a)}=T_{T_b(a)}$;
  \item $T_aS_a=S_aT_a=id$.
\end{enumerate}
\end{definition}

For example, let $S$ be a set, if for any $x\in S$ we have $x\ast y=x\circ y=x$ for all $y\in S$, then we call $S$ is a \emph{trivial virtual flat biquandle}. Moreover, for any bijection $\psi: S\rightarrow S$ we can define a \emph{constant-action virtual flat biquandle} structure on $S$ by setting $x\ast y=\psi(x)$ and $x\circ y=\psi^{-1}(x)$. The name ``constant-action" is borrowed from \cite{Hen2010}.

As another example, let us consider the virtual flat biquandle generated by one element, i.e. $<a>$. According to the virtual flat biquandle axioms it is not difficult to observe that the element of $<a>$ has the form $<a>=\{a, S_a^n(a), T_a^n(a)\}$ $(n\in \mathds{Z}^+)$.

Let $L$ be a flat virtual link diagram, the \emph{fundamental virtual flat biquandle} $VFB(L)$ is generated by the v-arcs $($here the term \emph{v-arc} means a part of the diagram from a virtual crossing to the next virtual crossing$)$ of the diagram under the equivalence relation generated by the virtual flat biquandle axioms and the relations at virtual crossings, see Figure 18.
\begin{center}
\includegraphics{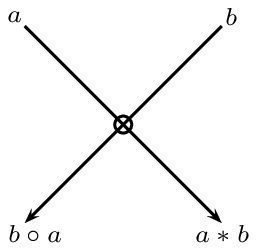}
\centerline{\small Figure 18}
\end{center}
\begin{theorem}
Let $L$ and $L'$ be a pair of flat virtual link diagrams. If $L$ can be transformed into $L'$ by one flat Reidemeister move, then there exists a virtual flat biquandle isomorphism between $VFB(L)$ and $VFB(L')$.
\end{theorem}
\begin{proof}
First we show that the operations of a virtual flat biquandle $VFB$ satisfy all the axioms of flat biquandle. This implies the invariance of $VFB(L)$ under flat $\Omega_1', \Omega_2', \Omega_3'$ \cite{Hen2010}.
\begin{enumerate}
\item $\forall a\in VFB$, if $x\ast a=a$, then
\begin{center}
$x\xlongequal{S_aT_a=id}(x\ast a)\circ a=a\circ a$.
\end{center}
On the other hand,
\begin{center}
$a\circ(a\circ a)\xlongequal{T_a=T_{T_b(a)}}a\circ a$
\end{center}
implies $a\circ x=x$. For the same reason one can prove that $y=a\ast a$.
\item $\forall a, b\in VFB$,
\begin{center}
$x\xlongequal{S_aT_a=id}(x\ast a)\circ a=b\circ a$,\\
$y\xlongequal{S_aT_a=id}(y\ast b)\circ b=a\circ b$.
\end{center}
In this case, it is easy to check that all other axioms are also satisfied.
\item $\forall a, b, c\in VFB$,
\begin{center}
$(a\circ b)\circ c\xlongequal{T_aT_b=T_bT_a}(a\circ c)\circ b\xlongequal{T_a=T_{S_b(a)}=T_{T_b(a)}}(a\circ(c\ast b))\circ(b\circ c)$.
\end{center}
The other two axioms can be verified in the same way.
 \end{enumerate}

Now it is sufficient for us to check the flat version of $\Omega_3^s$.
\begin{center}
\includegraphics{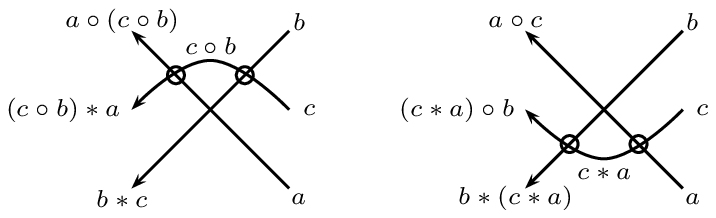}
\centerline{\small Figure 19}
\end{center}
Comparing the two figures in Figure 19, we have
\begin{center}
$a\circ(c\circ b)\xlongequal{T_a=T_{T_b(a)}}a\circ c, (c\circ b)\ast a\xlongequal{S_aT_b=T_bS_a}(c\ast a)\circ b$ and $b\ast c\xlongequal{S_a=S_{S_b(a)}}b\ast(c\ast a)$.
\end{center}
Other cases of flat $\Omega_3^s$ can be checked in the same manner. The proof is finished.
\end{proof}
\begin{corollary}
Let $L$ be a flat virtual link and $S$ a finite virtual flat biquandle. Then the cardinality of the set of virtual flat biquandle homomorphisms from $VFB(L)$ to $S$, denoted by $vc(L, S)$, is an invariant of $L$.
\end{corollary}

Similar to the quandle coloring invariant, the counting invariant described above also can be pictured as a coloring of $L$. Recall that a v-arc is a part of the diagram from a virtual crossing to the next virtual crossing. Then the counting invariant mentioned in Corollary 4.4 is equivalent to the number of assignments of an element of $S$ to each v-arc of $L$, such that at each virtual crossing the relation in Figure 18 is satisfied.

Here we give a simple example of the fundamental virtual flat biquandle and the counting invariant.
\begin{center}
\includegraphics{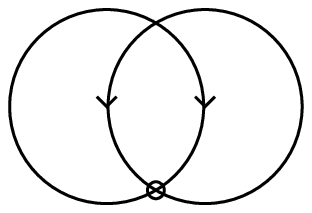}
\centerline{\small Figure 20}
\end{center}
Consider the flat virtual Hopf link $H$ in Figure 20. According to the definition of the fundamental virtual flat biquandle, we have
\begin{center}
$VFB(H)=<x, y|x\ast y=x, y\circ x=y>$.
\end{center}
On the other hand, it is obvious that $VFB(T)=<x, y>$, here $T$ denotes the trivial 2-component link. We claim that these two virtual flat biquandle are different. In order to see this, let us consider a set $S=\{x, y\}$ with a bijection $\psi(x)=y$ and $\psi(y)=x$, then $S$ is a constant-action virtual flat biquandle. Direct calculation shows that $vc(H, S)=0$ but $vc(T, S)=4$.

It is easy to show that a $\mathds{Z}[s]/(s^2)$-module $S$ with two binary operations $S_a(b)=-sa+(1-s)b$ and $T_a(b)=sa+(1+s)b$ is a virtual flat biquandle. This is exactly the algebraic structure we used at the beginning of this section, see Figure 16. In particular if $L$ is a $($flat$)$ virtual $n$-braid, by assigning $x_1, \cdots, x_n$ to the beginning of each strand one will obtain $n$ elements of the virtual flat biquandle that generated by $\{x_1, \cdots, x_n\}$ at the ends of these strands. This generalizes the matrix invariant discussed in Section 2. We remark that if we replace $\mathds{Z}[s]/(s^2)$ with a commutative ring $R$ without zero divisors, it was proved in \cite{Kau2012} that the general affine linear flat biquandle $($recall that a virtual flat biquandle is also a flat biquandle$)$ with coefficients in $R$ has the form $S_a(b)=\alpha b+k$ and $T_a(b)=\alpha^{-1}b-\alpha^{-1}k$, here $\alpha$ is an invertible element of $R$ and $k$ is an element of $S$. According to our discussion in Section 2, if we add two affine binary operations at each real crossing and another two affine binary operations at each virtual crossing of a virtual string link simultaneously, and assume we are working with a commutative ring $R$ without zero divisors, then it seems that the Burau representation is the unique meaningful solution.

\section{Cocycle invariants}
The cohomology theory of a rack was introduced by Fenn, Rourke and Sanderson in \cite{Fen1995,Fen1996}. Later the cohomology theory of a quandle was proposed by Carter, Jelsovsky, Kamada, Langford and Saito in \cite{Car2003}. With a fixed quandle and a 2-cocycle one can construct a state-sum invariant for knots. For biquandle, the cocycle invariants were defined using the Yang-Baxter cohomology theory \cite{Car2004}. In 2009 Ceniceros and Nelson introduced the virtual Yang-Baxter cocycle invariants for virtual biquandle \cite{Cen2009}. In this section we want to discuss the enhancement techniques for the counting invariant introduced in Section 4.

Given a virtual flat biquandle $S$, let $C_n(S)$ denote the free abelian group generated by $n$-tuples $(a_1, \cdots, a_n)$ . If $n=0$ we set $C_0(S)=0$. Consider the boundary map $\partial_n: C_n(S)\rightarrow C_{n-1}(S)$ defined by
\begin{center}
$\partial_n(a_1, \cdots, a_n)=\sum\limits_{i=1}^n(-1)^{i}((a_1\ast a_i, \cdots, a_{i-1}\ast a_i, a_{i+1}, \cdots, a_n)-(a_1, \cdots, a_{i-1}, a_{i+1}\circ a_i, \cdots, a_n\circ a_i))$
\end{center}
for $n\geq2$ and $\partial_n=0$ for $n\leq1$. We remark that the boundary map used by Ceniceros and Nelson in \cite{Cen2009} to define the $S$-homology is a special case of $\partial_n$. In fact when $S$ is a constant-action virtual flat biquandle the the boundary map defined above is exactly the one used in \cite{Cen2009}. On the other hand, assume $S$ is a virtual flat biquandle, it is an interesting exercise to check that the boundary map used in the homology theory of the set-theoretic Yang-Baxter equation \cite{Car2004} coincides with the boundary map $(-1)^{n+1}\partial_n$. Note that Lemma 5.1 below is an evident corollary of this fact, but we still give a direct proof here.
\begin{lemma}
$\partial^2=0$.
\end{lemma}
\begin{proof}
\begin{flalign*}
&\partial_{n-1}(\partial_n(a_1, \cdots, a_n))&\\
=&\partial_{n-1}(\sum\limits_{i=1}^n(-1)^{i}((a_1\ast a_i, \cdots, a_{i-1}\ast a_i, a_{i+1}, \cdots, a_n)-(a_1, \cdots, a_{i-1}, a_{i+1}\circ a_i, \cdots, a_n\circ a_i)))\\
=&\sum\limits_{j<i}(-1)^{i+j}((a_1\ast a_i)\ast a_j, \cdots, (a_{j-1}\ast a_i)\ast a_j, a_{j+1}\ast a_i, \cdots, a_{i-1}\ast a_i, a_{i+1}, \cdots, a_n)\\
&-\sum\limits_{j<i}(-1)^{i+j}(a_1\ast a_i, \cdots, a_{j-1}\ast a_i, (a_{j+1}\ast a_i)\circ a_j, \cdots, (a_{i-1}\ast a_i)\circ a_j, a_{i+1}\circ a_j, \cdots, a_n\circ a_j)\\
&+\sum\limits_{j>i}(-1)^{i+j+1}((a_1\ast a_i)\ast a_j, \cdots, (a_{i-1}\ast a_i)\ast a_j, a_{i+1}\ast a_j, \cdots, a_{j-1}\ast a_j, a_{j+1}, \cdots, a_n)\\
&-\sum\limits_{j>i}(-1)^{i+j+1}(a_1\ast a_i, \cdots, a_{i-1}\ast a_i, a_{i+1}, \cdots, a_{j-1}, a_{j+1}\circ a_j, \cdots, a_n\circ a_j)\\
&-\sum\limits_{j<i}(-1)^{i+j}(a_1\ast a_j, \cdots, a_{j-1}\ast a_j, a_{j+1}, \cdots, a_{i-1}, a_{i+1}\circ a_i, \cdots, a_n\circ a_i)\\
&+\sum\limits_{j<i}(-1)^{i+j}(a_1, \cdots, a_{j-1}, a_{j+1}\circ a_j, \cdots, a_{i-1}\circ a_j, (a_{i+1}\circ a_i)\circ a_j, \cdots, (a_n\circ a_i)\circ a_j)\\
&-\sum\limits_{j>i}(-1)^{i+j+1}(a_1\ast a_j, \cdots, a_{i-1}\ast a_j, (a_{i+1}\circ a_i)\ast a_j, \cdots, (a_{j-1}\circ a_i)\ast a_j, a_{j+1}\circ a_i, \cdots, a_n\circ a_i)\\
&+\sum\limits_{j>i}(-1)^{i+j+1}(a_1, \cdots, a_{i-1}, a_{i+1}\circ a_i, \cdots, a_{j-1}\circ a_i, (a_{j+1}\circ a_i)\circ a_j, \cdots, (a_n\circ a_i)\circ a_j)\\
=&0
\end{flalign*}
\end{proof}

Therefore $C_{\ast}(S)=\{C_n(S), \partial_n\}$ is a chain complex. Let $C_n'(S)$ be a subset of $C_n(S)$ generated by $(a_1, \cdots, a_i, a_{i+1}, \cdots, a_n)+(a_1, \cdots, a_{i+1}\circ a_i, a_i\ast a_{i+1}, \cdots, a_n)$ for $n\geq 2$, and $C_n'(S)=0$ for $n\leq 1$. Note that if $S$ is a trivial virtual flat biquandle, then $(a_1, \cdots, a_i, a_{i+1}, \cdots, a_n)+(a_1, \cdots, a_{i+1}\circ a_i, a_i\ast a_{i+1}, \cdots, a_n)$ reduces to the ``transposition symmetrizers" \cite{Prz2014}
\begin{center}
$(a_1, \cdots, a_i, a_{i+1}, \cdots, a_n)+(a_1, \cdots, a_{i+1}, a_i, \cdots, a_n)$.
\end{center}
\begin{lemma}
$C_{\ast}'(S)=\{C_n'(S), \partial_n\}$ is a sub-complex of $C_{\ast}(S)$.
\end{lemma}
\begin{proof}
It suffices to show that $\partial_n(C_n'(S))\subset C_{n-1}'(S)$. One computes
\begin{flalign*}
&\partial_n((a_1, \cdots, a_i, a_{i+1}, \cdots, a_n)+(a_1, \cdots, a_{i+1}\circ a_i, a_i\ast a_{i+1}, \cdots, a_n))&\\
=&\sum\limits_{k=1}^{i-1}(-1)^k(a_1\ast a_k, \cdots, a_{k-1}\ast a_k, a_{k+1}, \cdots, a_n)-\sum\limits_{k=1}^{i-1}(-1)^k(a_1, \cdots, a_{k-1}, a_{k+1}\circ a_k, \cdots, a_n\circ a_k)\\
&+(-1)^i(a_1\ast a_i, \cdots, a_{i-1}\ast a_i, a_{i+1}, \cdots, a_n)-(-1)^i(a_1, \cdots, a_{i-1}, a_{i+1}\circ a_i, \cdots, a_n\circ a_i)\\
&+(-1)^{i+1}(a_1\ast a_{i+1}, \cdots, a_i\ast a_{i+1}, a_{i+2}, \cdots, a_n)-(-1)^{i+1}(a_1, \cdots, a_i, a_{i+2}\circ a_{i+1}, \cdots, a_n\circ a_{i+1})\\
&+\sum\limits_{k=i+2}^{n}(-1)^k(a_1\ast a_k, \cdots, a_{k-1}\ast a_k, a_{k+1}, \cdots, a_n)-\sum\limits_{k=i+2}^{n}(-1)^k(a_1, \cdots, a_{k-1}, a_{k+1}\circ a_k, \cdots, a_n\circ a_k)\\
&+\sum\limits_{k=1}^{i-1}(-1)^k(a_1\ast a_k, \cdots, a_{k-1}\ast a_k, a_{k+1}, \cdots, a_{i+1}\circ a_i, a_i\ast a_{i+1}, \cdots, a_n)\\
&-\sum\limits_{k=1}^{i-1}(-1)^k(a_1, \cdots, a_{k-1}, a_{k+1}\circ a_k, \cdots, (a_{i+1}\circ a_i)\circ a_k, (a_i\ast a_{i+1})\circ a_k, \cdots, a_n\circ a_k)\\
&+(-1)^i(a_1\ast a_{i+1}, \cdots, a_{i-1}\ast a_{i+1}, a_i\ast a_{i+1}, \cdots, a_n)-(-1)^i(a_1, \cdots, a_i, a_{i+2}\circ a_{i+1}, \cdots, a_n\circ a_{i+1})\\
&+(-1)^{i+1}(a_1\ast a_i, \cdots, a_{i-1}\ast a_i, a_{i+1}, \cdots, a_n)-(-1)^{i+1}(a_1, \cdots, a_{i-1}, a_{i+1}\circ a_i, \cdots, a_n\circ a_i)\\
&+\sum\limits_{k=i+2}^{n}(-1)^k(a_1\ast a_k, \cdots, (a_{i+1}\circ a_i)\ast a_k, (a_i\ast a_{i+1})\ast a_k, \cdots, a_{k-1}\ast a_k, a_{k+1}, \cdots, a_n)\\
&-\sum\limits_{k=i+2}^{n}(-1)^k(a_1, \cdots, a_{i+1}\circ a_i, a_i\ast a_{i+1}, \cdots, a_{k-1}, a_{k+1}\circ a_k, \cdots, a_n\circ a_k)\\
=&\sum\limits_{k=1}^{i-1}(-1)^k((a_1\ast a_k, \cdots, a_{k-1}\ast a_k, a_{k+1}, \cdots, a_i, a_{i+1}, \cdots, a_n)\\
&+(a_1\ast a_k, \cdots, a_{k-1}\ast a_k, a_{k+1}, \cdots, a_{i+1}\circ a_i, a_i\ast a_{i+1}, \cdots, a_n))\\
&-\sum\limits_{k=1}^{i-1}(-1)^k((a_1, \cdots, a_{k-1}, a_{k+1}\circ a_k, \cdots, a_i\circ a_k, a_{i+1}\circ a_k, \cdots, a_n\circ a_k)\\
&+(a_1, \cdots, a_{k-1}, a_{k+1}\circ a_k, \cdots, (a_{i+1}\circ a_i)\circ a_k, (a_i\ast a_{i+1})\circ a_k, \cdots, a_n\circ a_k))\\
&+\sum\limits_{k=i+2}^{n}(-1)^k((a_1\ast a_k, \cdots, a_i\ast a_k, a_{i+1}\ast a_k, \cdots, a_{k-1}\ast a_k, a_{k+1}, \cdots, a_n)\\
&+(a_1\ast a_k, \cdots, (a_{i+1}\circ a_i)\ast a_k, (a_i\ast a_{i+1})\ast a_k, \cdots, a_{k-1}\ast a_k, a_{k+1}, \cdots, a_n))\\
&-\sum\limits_{k=i+2}^{n}(-1)^k((a_1, \cdots, a_i, a_{i+1}, \cdots, a_{k-1}, a_{k+1}\circ a_k, \cdots, a_n\circ a_k)\\
&+(a_1, \cdots, a_{i+1}\circ a_i, a_i\ast a_{i+1}, \cdots, a_{k-1}, a_{k+1}\circ a_k, \cdots, a_n\circ a_k))\\
\in&C_{n-1}'(S)
\end{flalign*}
\end{proof}

Let $C_{\ast}^{VF}(S)$ be the quotient complex $C_{\ast}(S)/C_{\ast}'(S)$ and $A$ an abelian group without 2-torsion, then we consider the homology and cohomology groups
\begin{center}
$H_n^{VF}(S; A)=H_n(C_{\ast}^{VF}(S)\otimes A)$ and $H^n_{VF}(S; A)=H^n(\text{Hom}(C_{\ast}^{VF}(S), A))$.
\end{center}
Before defining the cocycle invariants associated to virtual flat biquandle coloring, we need to introduce another boundary map of $C_{\ast}(S)$. As usual the notation $\widehat{a_i}$ denotes the removal of $a_i$. We define another boundary map $d_n: C_n(S)\rightarrow C_{n-1}(S)$ as below
\begin{center}
$d_n(a_1, \cdots, a_n)=\sum\limits_{i=1}^{n-1}(-1)^i((a_1, \cdots, \widehat{a_i}, \cdots, a_n)-(a_1, \cdots, \widehat{a_i}, \cdots, a_{n-1}, a_n\circ a_i))$.
\end{center}
\begin{lemma}
$d^2=0$.
\end{lemma}
\begin{proof}
\begin{flalign*}
&d_{n-1}(d_n(a_1, \cdots, a_n))&\\
=&d_{n-1}(\sum\limits_{i=1}^n(-1)^i((a_1, \cdots, \widehat{a_i}, \cdots, a_n)-(a_1, \cdots, \widehat{a_i}, \cdots, a_{n-1}, a_n\circ a_i)))\\
=&\sum\limits_{j<i}(-1)^{i+j}(a_1, \cdots, \widehat{a_j}, \cdots, \widehat{a_i}, \cdots, a_n)-\sum\limits_{j<i}(-1)^{i+j}(a_1, \cdots, \widehat{a_j}, \cdots, \widehat{a_i}, \cdots, a_{n-1}, a_n\circ a_j)\\
&+\sum\limits_{j>i}(-1)^{i+j+1}(a_1, \cdots, \widehat{a_i}, \cdots, \widehat{a_j}, \cdots, a_n)-\sum\limits_{j>i}(-1)^{i+j+1}(a_1, \cdots, \widehat{a_i}, \cdots, \widehat{a_j}, \cdots, a_{n-1}, a_n\circ a_j)\\
&-\sum\limits_{j<i}(-1)^{i+j}(a_1, \cdots, \widehat{a_j}, \cdots, \widehat{a_i}, \cdots, a_{n-1}, a_n\circ a_i)\\
&+\sum\limits_{j<i}(-1)^{i+j}(a_1, \cdots, \widehat{a_j}, \cdots, \widehat{a_i}, \cdots, a_{n-1}, (a_n\circ a_i)\circ a_j)\\
&-\sum\limits_{j>i}(-1)^{i+j+1}(a_1, \cdots, \widehat{a_i}, \cdots, \widehat{a_j}, \cdots, a_{n-1}, a_n\circ a_i)\\
&+\sum\limits_{j>i}(-1)^{i+j+1}(a_1, \cdots, \widehat{a_i}, \cdots, \widehat{a_j}, \cdots, a_{n-1}, (a_n\circ a_i)\circ a_j)\\
=&0
\end{flalign*}
\end{proof}

It follows that $C_{\ast}^{SF}(S)=\{C_n(S), d_n\}$ is a chain complex and we can define the the homology and cohomology groups
\begin{center}
$H_n^{SF}(S; A)=H_n(C_{\ast}^{SF}(S)\otimes A)$ and $H^n_{SF}(S; A)=H^n(\text{Hom}(C_{\ast}^{SF}(S), A))$.
\end{center}

Let $L$ be a flat virtual link, and $S$ a finite virtual flat biquandle. Assume we have a map $\phi: S\times S\rightarrow A$, then for a fixed coloring $\theta: VFB(L)\rightarrow S$ we can assign a (Boltzmann) weight to each virtual crossing of $L$. For example for the virtual crossing $\tau$ given in Figure 18 the associated weight  $B(\tau, \theta)=\phi(a, b)$. Then the state-sum $\Phi_{\phi}(L)$, which takes value in the group ring $\mathds{Z}[A]$, has the following expression
\begin{center}
$\Phi_{\phi}(L)=\sum\limits_{\theta}\prod\limits_{\tau}B(\tau, \theta)$,
\end{center}
here the product is taken over all virtual crossings and the sum is taken over all colorings. Note that the map $\phi: S\times S\rightarrow A$ can be naturally regarded as an element of $C^2_{VF}(S; A)$ or $C^2_{SF}(S; A)$ by linear extensions. For simplicity we will still use $\phi$ to denote it.
\begin{theorem}
Assume the map $\phi: S\times S\rightarrow A$ represents a 2-cocycle in both $C^2_{VF}(S; A)$ and $C^2_{SF}(S; A)$, then the state-sum $\Phi_{\phi}(L)$ is an invariant of $L$.
\end{theorem}
\begin{proof}
According to the assumption, the map $\phi: S\times S\rightarrow A$ satisfies the 2-cocycle conditions of $C^{\ast}_{VF}(S)$ and $C^{\ast}_{SF}(S)$ simultaneously. In other words, for any $a, b, c\in S$, $\phi$ should satisfy the following three conditions:
\begin{enumerate}
  \item $\phi(a, b)+\phi(b\circ a, a\ast b)=0$,
  \item $-\phi(b, c)+\phi(b\circ a, c\circ a)+\phi(a\ast b, c)-\phi(a, c\circ b)-\phi(a\ast c, b\ast c)+\phi(a, b)=0$,
  \item $-\phi(b, c)+\phi(b, c\circ a)+\phi(a, c)-\phi(a, c\circ b)=0$.
\end{enumerate}
From the first condition we conclude that (recall that $A$ has no 2-torsion)
\begin{center}
$\phi(a, a\ast a)+\phi((a\ast a)\circ a, a\ast(a\ast a))=0\Rightarrow\phi(a, a\ast a)+\phi(a, a\ast a)=0\Rightarrow\phi(a, a\ast a)=0$,
$\phi(a\circ a, a)+\phi(a\circ (a\circ a), (a\circ a)\ast a)=0\Rightarrow\phi((a\circ a, a)+\phi((a\circ a, a)=0\Rightarrow\phi((a\circ a, a)=0$,
\end{center}
and
\begin{center}
$\phi(b, a\ast b)+\phi((a\ast b)\circ b, b\ast(a\ast b))=0\Rightarrow\phi(b, a\ast b)+\phi(a, b\ast a)=0$,
$\phi(b\circ a, a)+\phi(a\circ(b\circ a), (b\circ a)\ast a)=0\Rightarrow\phi(b\circ a, a)+\phi(a\circ b, b)=0$.
\end{center}
The rest is a routine verification of the invariance of $\Phi_{\phi}(L)$ under flat Reidemeister moves. We only mention that the first condition and its ramifications guarantee the invariance of $\Phi_{\phi}(L)$ under the flat $\Omega_1'$ and $\Omega_2'$. On the other hand, the second and the third condition make $\Phi_{\phi}(L)$ invariant under one case of flat $\Omega_3'$ and $\Omega_3^s$ respectively. Other cases of flat $\Omega_3'$ and $\Omega_3^s$ can be realized by these two cases together with some flat $\Omega_1'$ and $\Omega_2'$. We leave the details to the reader.
\end{proof}

We give a simple example to show that sometimes the cocycle invariant is stronger than the counting invariant. Consider the flat virtual link $H$ in Figure 20. Let $S$ be the trivial virtual flat biquandle with two elements, i.e.
\begin{center}
$S=\{x, y|x\ast x=x\circ x=x\ast y=x\circ y=x, y\ast x=y\circ x=y\ast y=y\circ y=y\}$.
\end{center}
As before we use $T$ to denote the trivial 2-component link, then we have $vc(H, S)=vc(T, S)=4$. Hence the counting invariant associated with $S$ can not distinguish $H$ from $T$. However the cocycle invariant associated with $S$ can tell the differences between $H$ and $T$. In order to see this, let us consider the map $\phi: S\times S\rightarrow \mathds{Z}$ defined by
\begin{center}
$\phi(x, y)=-\phi(y, x)=1$ and $\phi(x, x)=\phi(y, y)=0$.
\end{center}
It satisfies the conditions in Theorem 5.4. Direct calculation shows that $\Phi_{\phi}(H)=1+(-1)+0+0$ but $\Phi_{\phi}(T)=0+0+0+0$.

In the classical case \cite{Car2003}, the quandle cocycle invariants are actually cohomology invariants, i.e. cohomologous cocycles define the same invariant. However this is not the case for the virtual Yang-Baxter 2-cocycle invariant defined in \cite{Cen2009}. In fact, coboundaries can contribute nontrivially to the cocycle invariant in \cite{Cen2009}. For virtual flat biquandle we have the following result.
\begin{proposition}
Let $\phi$ be a 2-cocycle of $C^2_{VF}(S; A)$ and $C^2_{SF}(S; A)$, and $\phi'$ a 2-coboundary of $C^2_{VF}(S; A)$. Then $\phi+\phi'$ represents a 2-cocycle of $C^2_{SF}(S; A)$ and $\Phi_{\phi}(L)=\Phi_{\phi+\phi'}(L)$ for any flat virtual link $L$.
\end{proposition}
\begin{proof}
First we show that $\phi+\phi'$ represents a 2-cocycle of $C^2_{SF}(S; A)$, it suffices to prove that $\phi'$ represents a 2-cocycle of $C^2_{SF}(S; A)$. For this purpose it is sufficient to check that $\partial_2d_3=0$. In fact
\begin{flalign*}
&\partial_2d_3(a, b, c)&\\
=&\partial_2(-(b, c)+(b, c\circ a)+(a, c)-(a, c\circ b))\\
=&(c)-(c\circ b)-(b\ast c)+(b)-(c\circ a)+((c\circ a)\circ b)+(b\ast c)-(b)\\
&-(c)+(c\circ a)+(a\ast c)-(a)+(c\circ b)-((c\circ b)\circ a)-(a\ast c)+(a)\\
=&0\\
\end{flalign*}

Next we show that $\Phi_{\phi}(L)=\Phi_{\phi+\phi'}(L)$. According to the definition it suffices to show that $\Phi_{\phi'}(L)=\sum\limits_{\theta}0$. For a fixed coloring $\theta$, assuming the 2-coboundary $\phi'=\delta\eta$, where $\delta$ denotes the coboundary map and $\eta\in C^1_{VF}(S; A)$. Let us consider the virtual crossing in Figure 18, the contribution of this virtual crossing is
\begin{center}
$\phi(a, b)=\delta\eta(a, b)=\eta(\partial_2(a, b))=-\eta(b)+\eta(b\circ a)+\eta(a\ast b)-\eta(a)$.
\end{center}
It is easy to observe that the contributions of all virtual crossings will cancel out. The proof is finished.
\end{proof}

\end{document}